\documentclass[11pt, reqno]{amsart}

\usepackage{amsmath,amssymb,amsthm,amsfonts,verbatim}
\usepackage{microtype}
\usepackage[all,2cell]{xy}
\usepackage{mathtools}
\usepackage{graphicx}
\usepackage{pinlabel}
\usepackage{hyperref}
\usepackage{mathrsfs}
\usepackage{color}
\usepackage[dvipsnames]{xcolor}
\usepackage{enumerate}
\usepackage{cite}
\usepackage{soul}

\CompileMatrices

\usepackage[top=1.2in,bottom=1.2in,left=1in,right=1in]{geometry}

\usepackage{hyperref} 
\hypersetup{          
	colorlinks=true, breaklinks, linkcolor=[RGB]{51 102 204}, filecolor=Orchid, urlcolor=[RGB]{51 102 204},
	citecolor=Orchid, linktoc=all, }
\usepackage[nameinlink]{cleveref}

\theoremstyle{plain}
\newtheorem{theorem}{Theorem}[section]
\newtheorem{maintheorem}{Theorem}

\newtheorem{proposition}[theorem]{Proposition}
\newtheorem{lemma}[theorem]{Lemma}

\newtheorem{corollary}[theorem]{Corollary}

\theoremstyle{definition}
\newtheorem{definition}[theorem]{Definition}

\newtheorem{question}[theorem]{Question}

\newcommand{\nc}{\newcommand}
\nc{\dmo}{\DeclareMathOperator}
\newcommand{\R}{\mathbb{R}}
\newcommand{\z}{\mathbb{Z}}

\newcommand{\Ho}{\mathcal{HO}}
\newcommand{\h}{\mathbb{H}^2}

\nc{\Q}{\mathbb{Q}}
\nc{\F}{\mathbb{F}}
\nc{\Z}{\mathbb{Z}}
\nc{\C}{\mathbb{C}}
\nc{\N}{\mathbb{N}}
\nc{\Ell}{\mathcal{L}}
\nc{\M}{\mathcal{M}}
\nc{\K}{\mathcal{K}}
\nc{\I}{\mathcal{I}}
\nc{\T}{\mathcal T}
\nc{\U}{\mathcal U}
\nc{\disk}{\mathbb{D}}
\nc{\hyp}{\mathbb{H}}

\nc{\CP}{\mathbb{CP}}
\nc{\RP}{\mathbb{RP}}
\dmo{\Mod}{Mod}
\dmo{\PMod}{PMod}
\dmo{\LMod}{LMod}
\dmo{\Diff}{Diff}
\dmo{\Homeo}{Homeo}
\dmo{\dist}{dist}
\dmo\BDiff{BDiff}
\dmo\SO{SO}
\dmo\Hom{Hom}
\dmo\SL{SL}
\dmo\rank{rank}
\dmo\sig{sig}
\dmo\Out{Out}
\dmo\Aut{Aut}
\dmo\Inn{Inn}
\dmo\GL{GL}
\dmo\PGL{PGL}
\dmo\Gr{Gr}
\dmo\PSL{PSL}
\dmo\BHomeo{BHomeo}
\dmo\EHomeo{EHomeo}
\dmo\EDiff{EDiff}
\dmo\Disc{Disc}
\dmo\Aff{Aff}

\renewcommand{\hat}{\widehat}

\dmo\Teich{Teich}
\dmo\Fix{Fix}
\nc{\pair}[1]{\ensuremath{\left\langle #1 \right\rangle}}
\nc{\abs}[1]{\ensuremath{\left| #1 \right|}}
\nc{\action}{\circlearrowright}
\nc{\norm}[1]{\left | \left | #1 \right | \right |}
\nc{\abcd}[4]{\ensuremath{\left(\begin{array}{cc} #1 & #2 \\ #3 & #4 \end{array}\right)}}
\nc{\into}\hookrightarrow
\dmo{\Isom}{Isom}
\nc{\normal}{\vartriangleleft}
\dmo{\Vol}{Vol}
\dmo{\im}{Im}
\dmo{\Push}{Push}
\dmo{\Conf}{Conf}
\dmo{\PConf}{PConf}
\dmo{\PB}{PB}
\dmo{\id}{id}
\dmo{\Jac}{Jac}
\dmo{\Pic}{Pic}
\dmo{\Stab}{Stab}
\dmo{\Arf}{Arf}
\dmo{\End}{End}
\dmo{\Gal}{Gal}
\dmo{\lcm}{lcm}
\dmo{\ab}{ab}
\dmo{\opp}{op}
\dmo{\SU}{SU}
\dmo{\OT}{\Omega \mathcal{T}}
\dmo{\OM}{\Omega \mathcal{M}}
\dmo{\PH}{\mathbb{P}\mathcal{H}}
\dmo{\spin}{spin}
\dmo{\even}{even}
\dmo{\odd}{odd}
\dmo{\comp}{\mathcal{H}}
\dmo{\Mgk}{\mathcal{M}_{g, \underline{\kappa}}}
\dmo{\orb}{orb}
\dmo{\AJ}{AJ}
\dmo{\Ck}{\mathsf{C}(\underline{\kappa})}
\dmo{\Int}{Int}
\dmo{\pr}{pr}
\dmo{\lab}{lab}
\dmo{\Sym}{Sym}
\dmo{\Ann}{Ann}
\dmo{\Rad}{Rad}
\dmo{\Ind}{Ind}
\dmo{\Div}{Div}
\dmo{\Res}{Res}
\dmo{\Hur}{Hur}
\dmo{\vcd}{vcd}

\nc{\Span}[1]{\operatorname{Span}(#1)}

\renewcommand{\epsilon}{\varepsilon}
\renewcommand{\tilde}{\widetilde}

\nc{\coloneq}{\mathrel{\mathop:}\mkern-1.2mu=}
\nc{\margin}[1]{\marginpar{\scriptsize #1}}
\nc{\para}[1]{\medskip\noindent\textbf{#1.}}
\definecolor{myblue}{RGB}{102,153, 255}
\definecolor{myred}{RGB}{204,0,0}
\definecolor{mygreen}{RGB}{0,204,0}
\definecolor{myorange}{RGB}{255,102,0}
\definecolor{mypurple}{RGB}{138,43,226}
\nc{\red}[1]{\textcolor{myred}{#1}}
\nc{\blue}[1]{\textcolor{myblue}{#1}}

\title[Rigidity on horocycles and hypercycles of $\h$]{Rigidity on horocycles and hypercycles of $\h$}
\author{Cheikh Lo, Abdoul Karim Sane}
\date{May 26, 2024}

\begin{document}

\maketitle
\begin{abstract} We show that a bijection $f:\h\rightarrow\h$ of the hyperbolic plane that sends horocycles to horocycles (respectively hypercycles to hypercycles) is an isometry. This extends a previous result of J. Jeffers on geodesics to all curves with constant curvature in $\h$. We go beyond by showing that every abstract automorphism of the geodesic  graph (respectively horocycles and hypercycles graphs) is induced by an earthquake map (respectively an isometry) of $\h$. This shadowed the difference between the geometry of geodesics and that of horocycles/hypercycles. 
\end{abstract}

\section{Introduction}
In this article, we prove two kinds of results. The first one is inspired by works of Jeffers \cite{Jef} where he showed, among many other similar results, that geodesic-preserving bijections of the hyperbolic plane $\h$ are isometries.  The second one is about the automorphisms of the graph associated to objects on a given geometry: an Ivanov-like theorem.\vspace{0.2cm} 
\begin{paragraph}{\textbf{Jeffers-type results:}} In \textit{Lost theorem of geometry \cite{Jef}}, J. Jeffers showed that geodesic-preserving bijections of $\h$ (respectively $\mathbb{R}^2$) are isometries (respectively affine map). This raises the question of finding the geometries that have this kind of results. Let $G$ be a group and $X$ be a set with a $G$-action; we say that $X$ is a $G$-geometry in the sense of Thurston. A set $\mathcal{O}$ of object in $X$ stable under the action of $G$ is a set of \textit{$G$-geometric object}. For instance when $X=\h$ and $G$ is the group of M\"{o}bius transformation, one can choose $\mathcal{O}$ to be the set of all geodesics. 
\begin{question}\label{qst} For which pairs $[(X,G),\mathcal{O}]$, where $X$ is a $G$-geometry and $\mathcal{O}$ is a set of geometric objects on $X$, we have a Jeffers-like result? In other words, every bijection $f:X\rightarrow X$ preserving elements of $\mathcal{O}$ is an element of $G$.
\end{question}
 
Aside from  well-known geometric objects on $\h$ like geodesics and hyperbolic circles, we also have \textit{horocycles} (respectively \textit{hypercycles}) which correspond to orbit flows of parabolic isometries (respectively hyperbolic isometries); see \cite{Bea}.

In this article we first extend Jeffers result to the family of horocycles. Our theorem is:
\begin{maintheorem}\label{main} Let $f:\h\rightarrow\h$ be a bijection that sends horocycles to horocycles. Then, $f$ is an isometry. 
\end{maintheorem}

We also prove a similar result for hypercycles. We recall that by definition, points on a hypercycle are at the same distance to a given geodesic and this property characterizes hypercycles. 
\begin{maintheorem}\label{main2} Let $f:\h\rightarrow\h$ be a bijection that sends hypercycles to hypercycles. Then, $f$ is an isometry.
\end{maintheorem}
\end{paragraph} 
Since geodesics (respectively circles, horocycles and hypercycles) have constant curvature $r=0$ (respectively $r\in(1,+\infty)$, $r=1$ and $r\in(0,1)$), our theorem together with J. Jeffers theorems answer \Cref{qst} for all curves with constant curvature in $\mathbb{H}^2$. The main idea that is used to  prove these results is to show that postcomposing a given geodesic-preserving bijection (respectively circle-preserving, horocycle-preserving, hypercycle-preserving) leads to the identity map. One of the major obstacles we overcame in our proof of \Cref{main} and \Cref{main2} is that horocycles and hypercycles are not completly determined by their endpoints; which was central in the geodesic case.    
Another contrast between geodesics/circles and horocycles/hypercycles is that some of the results proved in \cite{Jef} are not true for horocycles. For instance a bijection $f$ on the set $\mathcal{C}$ of all circles in $\h$ that preserves inclusion of circles, induces a bijection of $\h$ that sends circles to circles; thus $f$ is induced by a M\"{o}bius transformation using Theorem 5.1  in \cite{Jef}. It contrasts with the situation where we consider the set $\Ho$ of all horocylces of $\h$ endowed with the partial order induced by inclusion:
\begin{equation}\label{equ} 
h_1\leq h_2\;\;\mbox{if the disk bounded by}\;h_2\;\mbox{contains}\;h_1.
\end{equation}
In particular $h_1\leq h_2$ implies that $h_1$ and $h_2$ have the same center in $\h$; the center of a horocycle $h$ referred to the unique point of $h\cap\partial\h$. 
 Here is a counter-example of a map on the set of all horocycles that preserves inclusion without being an isometry. Let $\Ho(p)$ denotes the set of all horocycles centered at ~$p$. We denote by $h(p,r)$ the horocycle centered at $p$ with Euclidean radius equal $r$ (when $p=\infty$, $h(p,r)$ is a horizontal line and $r$ is the imaginary part of all its points). Now, let $\sigma_{p,q}:\Ho\rightarrow\Ho$ be the map defined by $\sigma_{p,q}(h(p,r))=h(q,r),\sigma_{p,q}(h(q,r))=h(p,r)$ and $\sigma_{p,q}(h(x,r))=h(x,r)$ for any $x\in\R-\{p,q\}$. One can see that $\sigma_{p,q}$ is a bijection that preserves the order while $\sigma_{p,q}$ is not induced by an isometry since its restriction to $\partial\h$ is not continuous. In general, any bijection $\sigma$ of $\partial\h$ induces a such map on ~$\Ho$.  

\begin{paragraph}{\textbf{Ivanov-like results:}}These results fit into many others and follow the idea that geometry can be encoded by a group action. On the other hand, Jeffers result on geodesics together with our results have a stronger version which can be stated as a Ivanov-like theorem (see \cite{Papa} for survey on Ivanov metaconjecture). To do so, let $\mathcal{K}_{geod}$ (respectively $\mathcal{K}_{horo}$ and $\mathcal{K}_{hyper}$) be the graph of geodesics (respectively horocycles and hypercyles) in $\h$ where the vertices correspond to geodesics (respectively horocycles and hypercycles) and the edges correspond to disjointness. We aim to describe the automorphism group of these graphs. Let $t$, $x$, $y$ and $z$ be four different points on $\mathbb{S}^1$ considered as $\partial\h$. The sets $\{t,x\}$ and $\{y,z\}$ are \textit{linked} if $y$ and $z$ are in different components of $\mathbb{S}^1-\{t,x\}$. A map $f:\mathbb{S}^1\rightarrow\mathbb{S}^1$ preserves links if it sends linked sets to linked sets. Since every $f\in\mathrm{Homeo}(\mathbb{S}^1)$ is either orientation-preserving or reversing, it turns out that $f$ is link-preserving. Since a geodesic is completely defined by its endpoints, every element $f\in\mathrm{Homeo}(\mathbb{S}^1)$ induces an automorphism $\hat{f}$ of $\mathcal{K}_{geod}$ defined by $\hat{f}:g:=\{x,y\}\mapsto \{f(x),f(y)\}$ where $g$ is the geodesic defined by $x$ and $y$ in $\mathbb{S}^1$. So, there is a natural injection $i:\mathrm{Homeo}(\mathbb{S}^1)\rightarrow\mathrm{Aut}(\mathcal{K}_{goed})$. Unlike geodesics, horocycles and hypercycles are not determined by their endpoints and this implies different behaviors as shown below: 

\begin{maintheorem}\label{ivan} We consider the graph $\mathcal{K}_{geod}$ (respectively $\mathcal{K}_{horo}$ and $\mathcal{K}_{hyper}$) of geodesics (respectively horocycles and hypercycles). 
\begin{itemize}
\item The natural injection $i:\mathrm{Homeo}(\mathbb{S}^1)\rightarrow\mathrm{Aut}(\mathcal{K}_{geod})$ is an isomorphism. In other words, every automorphisms of $\mathcal{K}_{geod}$ is induced by a link-preserving homeomorphism. 
\item The natural injection $i:\mathrm{Isom}(\h)\rightarrow\mathrm{Aut}(\mathcal{K}_{horo})$ is an isomorphism and the same result also happens for hypercycles. 
\end{itemize}
\end{maintheorem}

\Cref{ivan} suggests that horocycles and hypercycles are more rigid than geodesics in some sense, since their automorphisms group are smaller than the one for geodesics. This comes from the fact that the $\mathcal{K}_{horo}$ (respectively $\mathcal{K}_{hyper}$) is richer than $\mathcal{K}_{geod}$ since horocycles (respectively hypercycles) can be tangent or can intersect twice in $\h$. These facts endow $\mathcal{K}_{horo}$ and $\mathcal{K}_{hyper}$ with much more properties.\\
Every isometry $f$ induces a map $\partial{f}$ on the boundary: \textit{a boundary map}. Therefore, the group $\partial\mathrm{Isom}(\h)$ of boundary maps induced by isometries is a subgroup of $\mathrm{Homeo}(\mathbb{S}^1)$. Moreover, by postcomposing an element $f\in\mathrm{Homeo}(\mathbb{S}^1)$ by an element in $\partial\mathrm{Isom}(\h)$, we get an orientation-preserving homeomorphism of $\mathbb{S}^1$; and by Thurston earthquake theorem such homeomorphisms are induced by so called \textit{\textbf{earthquake maps}} on $\h$ (see \cite{Hu} for more details on earthquake maps and Thurston's earthquake theorem).
\begin{figure}[htbp]
\begin{center}
\includegraphics[scale=0.4]{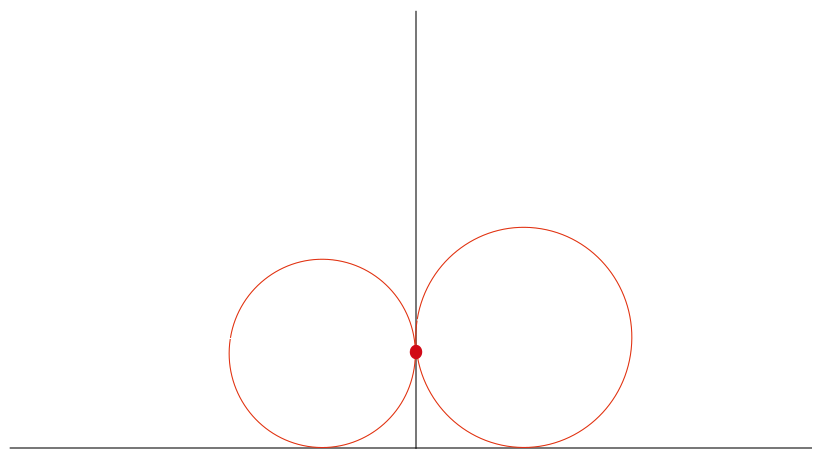}\put(10,30){$\longrightarrow$}\hspace{1,7cm}\includegraphics[scale=0.4]{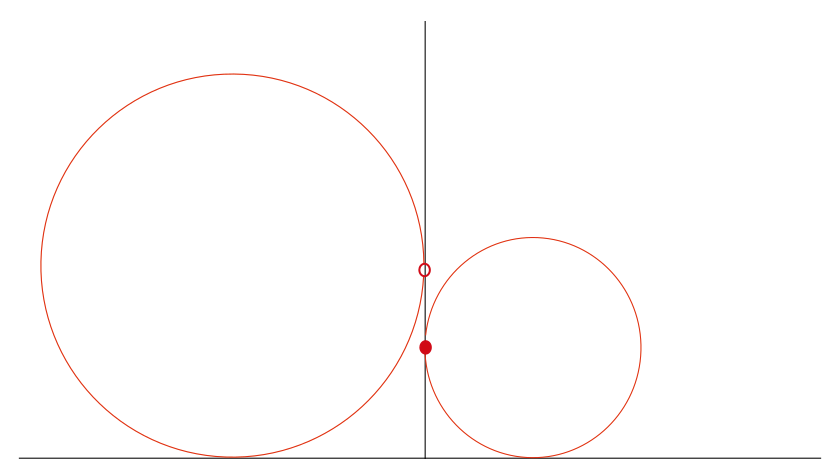}
\caption{An earthquake map $f:\h\rightarrow\h$ defined by $f(z)=2z$ if $\mathrm{Re}(z)<0$ and $f(z)=z$ otherwise. The map $f$ sends tangent horocycles to disjoint ones and hence do not preserved the graph structure. Indeed, the map $f$ does not send horocycles to horocycles (nor does it for geodesics). Since a horocycle is determined by its center and its radius, what \Cref{ivan} tells us is that there is no way for $f$ to induce a radius fuction $\tilde{f}$ such that $\hat{f}:\mathcal{HO}\rightarrow\mathcal{HO}, h(x,r)\mapsto h(\partial{f}(x),\tilde{f}(r))$ preserves the graph structure of $\mathcal{K}_{horo}$.}
\label{quake}
\end{center}
\end{figure}

 So, an earthquake map in $\h$ induced automorphism on an $\mathcal{K}_{geod}$ but not on $\mathcal{K}_{horo}$ nor on $\mathcal{K}_{hyper}$. One can easily guess that an earthquake is more likely to break tangencies between horocycles (respectively hypercycles) than transverse intersections between geodesics (see \Cref{quake} for an example of a simple earthquake and how it breaks tangencies on horocycles). And this what \Cref{ivan} tells us.  The difference between automorphisms induced by earthquake maps and those induced by isometries is that the latter sends geodesics that intersect at a single point to geodesics with the same property. So every automorphism $\hat{f}$ of $\mathcal{K}_{geod}$, which sends geodesics intersecting at a single point to geodesics of the same type is induced by an isometry.\vspace{0.2cm} 
\end{paragraph}

\begin{paragraph}{\textbf{Overview}} The proof of \Cref{main} relies on three facts. First we show that a bijection $f:\h\rightarrow\h$ that sends horocycles to horocycles extends to the boundary of $\h$. And then, we prove that after postcomposing with an isometry, one can show that $f$ fixes pointwise $\partial \h$ and the set $L:=\{z\in\h, \mathrm{Im}(z)=1\}$. The proof ends up by showing that if $f$ fixes pointwise $\partial\h$ and a horizontal line $L_y:=\{z\in\h, \mathrm{Im}(z)=y\}$ then $f(z)=z$ whenever $\mathrm{Im}(z)\neq y$.  For \Cref{main2} we show that a hypercycle-preserving bijection $f$ has to send geodesic to geodesic and Jeffers theorem implies that it is an isometry.  

Finally, we will show \Cref{ivan}. For the case of geodesics, we show that an automorphism $\hat{f}$ of $\mathcal{K}_{geod}$ induced a boundary map $\partial\hat{f}\in\mathrm{Homeo}(\mathbb{S}^1)$. In the case of horocycles we use postcomposition techniques to show that a given automorphism comes from an isometry. We end the proof for hypercycles by reduction to the horocycles case namely by showing that a non trivial element $\hat{f}\in\mathrm{Aut}(\mathcal{K}_{hyper})$ induces a non trivial element $\hat{f}^*\in\mathrm{Aut}(\mathcal{K}_{horo})$.  
\end{paragraph}\vspace{0.2cm}
\begin{paragraph}{\textbf{Acknowledgments:}} The athours was inspired by Benson Farb talk on reconstruction problems \cite{Farb}. They would like to thank Dan Margalit for comments on this work. The authors are grateful to Katherine Williams Booth, Ryan Dickmann, Masseye Gaye, Abdou Aziz Diop and Amadou Sy for their comments and the interest shown to this work.       
\end{paragraph}
\section{Proof of \Cref{main}}
The proof of \Cref{main} relies on two big facts. First we show that a horocycle-preserving bijection extends to a bijection on the boundary of $\h$ (see \Cref{projet}). Then, we show that up to postcomposition by isometries, it fixes pointwise $\partial\h$ and $L:=\{z:=x+i\}$ (see \Cref{ctr}) which helps to end the proof. 
\\We recall some basic facts about horocycles that will be useful for what follows.
\begin{itemize}
\item If $x\in\partial\h$ and $z\in\h$, there is a unique horocycle, denoted $h_z(x)$, centered at $x$ and passing through $z$.
\item If $z_1$ and $z_2$ are two distinct points in $\h$, there are exactly two horocycles passing through $z_1$ and $z_2$.\item Let $z_1$, $z_2$ and $z_3$ be three distinct points with $(z_1, z_2)\in\h\times\h$ and $z_3\in\h\cup\partial\h$.  There is at most one horocycle passing through $z_1$, $z_2$ and $z_3$. 
\end{itemize} 

Now, let $f:\h\rightarrow\h$ be a bijection such that for all $h\in\Ho$, $f(h)\in~\Ho$. Two distinct horocycles intersect at most twice and the intersection pattern between horocycles is preserved by $f$. We aim to show that $f$ is an isometry.

\begin{paragraph}{\textbf{Extension of $f$ to $\partial\h$.}} Recall for every $p\in\partial\h$, $\Ho(p)$ denotes the set of all horocycles centered at $p$.  The partial order (\ref{equ}) on $\Ho$ induces a total order on ~$\Ho(p)$. The following equivalence gives another way to see that partial order in terms of intersections:
$$h_1\leq h_2 \iff \forall h\in\Ho, h\cap h_1\cap\h\neq\emptyset \implies h\cap h_2\cap\h\neq\emptyset.$$
This follows from the Jordan theorem since $h_2\in\Ho(p)$ separates $\h\cup\partial\h$ into two components one of which contains $h_1$. As a consequence, we have:  
\begin{lemma}\label{jet}
Let $h_1$ and $h_2$ be two elements of $\Ho(p)$ such that $h_1\leq h_2$. Then $f(h_1)$ and $f(h_2)$ have the same center and the order is preserved on the images: $f(h_1)\leq f(h_2)$.  
\end{lemma}
\begin{proof}
Assume that $h_1\leq h_2$ while $f(h_1)$ and $f(h_2)$ have different centers. Then, there exists $h\in\Ho$ such that $h\cap f(h_1)\cap\h\neq \emptyset$ and $h\cap f(h_2)\cap\h=\emptyset$. Therefore $f^{-1}(h)\cap h_1\cap\h\neq\emptyset$ and $f^{-1}(h)\cap h_2\cap\h=\emptyset$ which is absurd. Using the same technique, we can show that the order is preserved.  
\end{proof}
\begin{figure}[htbp]
\begin{center}
\includegraphics[scale=0.5]{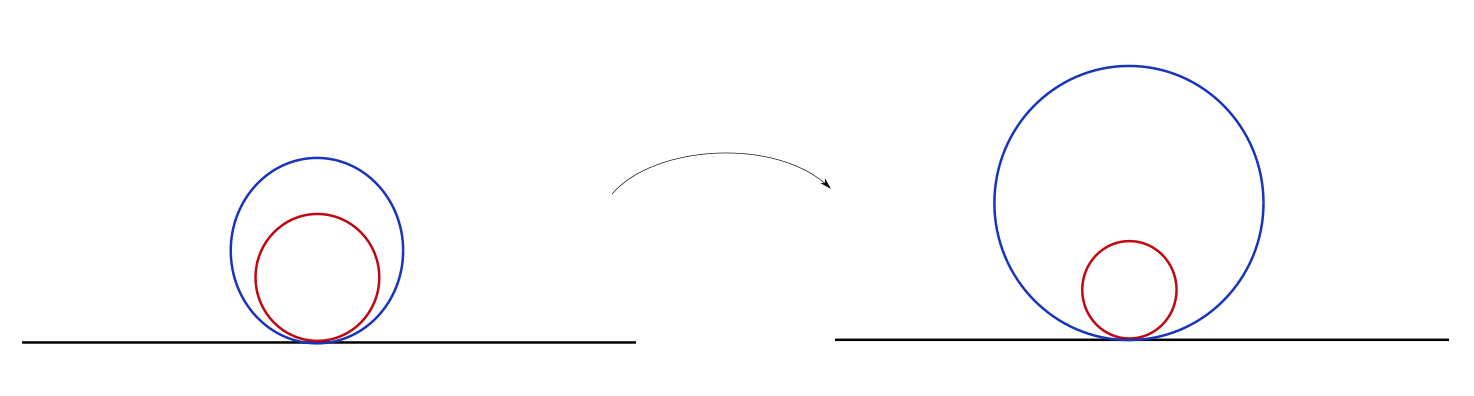}
\put(-190,70){$f$}
\put(-295,40){\tiny{$h_1$}}
\put(-295,63){\tiny{$h_2$}}
\put(-95,41){\tiny{$f(h_1)$}}
\put(-95,85){\tiny{$f(h_2)$}}
\caption{Image of two horocycles with the same center}
\label{default}
\end{center}
\end{figure}
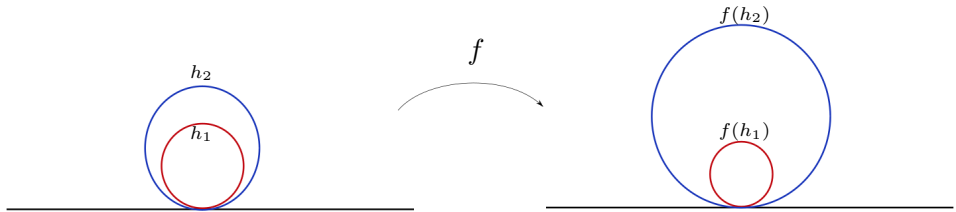
\begin{proposition}\label{projet} Let $f:\h\rightarrow\h$ be a horocycle-preserving bijection. Then $f$ extends to a well-defined bijection $f:\partial\h\rightarrow\partial\h$. 
\end{proposition} 
\begin{proof} Let $x\in\partial\h$ and $h_z(x)$ be any horocycle centered at $x$. We set $f(x):=y$ where $y$ is the center of $f(h_z(x))$. By \Cref{jet}, $f(x)$ depends only on $x$; so it is a well-defined map. Moreover, if $f(x_1)=f(x_2)$, then $f(h_{z_1}(x_1))$ and $f(h_{z_2}(x_2))$ have the same center and one can assume that $f(h_{z_1}(x_1))<f(h_{z_2}(x_2))$. Since $f$ also preserves order, it follows that $h_{z_1}(x_1)<h_{z_2}(x_2)$ which implies that $x_1=x_2$. Thus, $f$ is a bijection.   
\end{proof}
\end{paragraph}

\begin{paragraph}{\textbf{Postcomposition and fixed sets.}} The action of $\mathrm{Isom}(\h)$ on $\h$ extends to $\partial\h$ and is $2$-transitive on $\partial\h$. In fact, if $x$ and $y$ are two distinct points in $\R$, $\phi:z\mapsto\frac{1}{z-x}-\frac{1}{y-x}$ is an isometry of $\h$ and $\phi(x)=\infty$, $\phi(y)=0$. So, after postcomposing $f$ with an isometry, we can assume that $f(0)=0$ and $f(\infty)=\infty$.
\begin{lemma}
Let $f:\h\rightarrow\h$ be a geodesic-preserving bijection. Then, there exists an isometry $j$ such that $(j\circ f)(n)=n$ for all $n\in\mathbb{Z}$. 
\end{lemma}
 \begin{proof} Since the action of $\mathrm{Isom}(\h)$ is 2-transitive, there exists an isometry $\phi_1$ such that $\phi_1(f(0))=0$ and $\phi_1(f(\infty))=\infty$. Therefore, $\phi_1\circ f$ fixes $0$ and $\infty$. Now, let $h_i(\infty)$ be the horocycle centered at $\infty$ and passing through $i$ and $h_i(0)$ be the one centered at $0$ and passing through ~$i$. Since $\phi_1\circ f$ fixes $0$ and $\infty$, $(\phi_1\circ f)(h_i(\infty))$ and $(\phi_1\circ f)(h_i(0))$ are centered at $\infty$ and $0$ respectively, and tangent at a point ~$iy$. It follows that $\phi_1\circ f(h_i(\infty))=h_{iy}(\infty)$ and $\phi_1\circ f(h_i(0))=h_{iy}(0)$. Again, by postcomposing with $\phi_2:z\rightarrow \frac{z}{y}$ we have that $f':=\phi_1\circ \phi_2\circ f$ fixes $0$, $i$ and $\infty$ and thereby we have $f'(h_i(0))=h_i(0)$ and $f'(h_i(\infty))=h_i(\infty)$. 
Let $(h_n)_{n\in\z}$ be the sequence of horocycles where $h_n:=h_{n+i}(n)$ is centered at $n$ and passes through $n+i$ (see Figure ~\ref{seq}).

\begin{figure}[htbp]
\begin{center}
\includegraphics[scale=0.8]{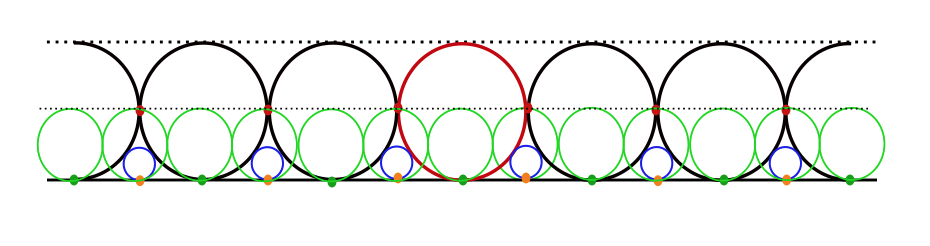}
\caption{The sequence $(h_n)_{n\in\z}$ made by horocycles centered at integer point with ecludiean diameter equal 1; where two consecutive horocycles are tangent. More generally, for every $k\in\mathbb{N}$, the sequence $(h^k_n)_{n\in\z}$ is made by geodesics centered at $\mathcal{D}_k$ with Euclidean diameter equal $\frac{1}{2^k}$ with consecutive horocycles that are tangent.}  
\label{seq}
\end{center}
\end{figure}
 Since $h_i(\infty)$ is fixed by $f$ and $f'(h_0)=h_0$ then $f'(h_1)$ is either equal to $h_1$ or $h_{-1}$ using the fact that $f'(h_1)$ is tangent to $h_0$ and $h_i(\infty)$. By postcomposing again with a reflection $r$ along $(Oy)$ we have that $h_1$ is fixed setwise by $f":=r\circ \phi_2\circ \phi_1\circ f$. Therefore ~$h_n$ is fixed setwise by $f"$ for all $n\in\z$ since the tangency pattern between the horocylces is preserved by $f"$. So, $j\circ f(m)=m$ for all $m\in\z$; where $j:=r\circ \phi_2\circ \phi_1$.
 \end{proof}
 
  Now, we can assume without lost of generality that $f:\h\rightarrow\h$ is a horocycle-preserving bijection such that fixes pointwise $\mathbb{Z}$.  
\\ Let $\mathcal{D}_k$ be the sequence of subsets of $\R$ defined by induction as follow:
 \begin{itemize}
 \item $\mathcal{D}_0=\z$,
 \item if $\mathcal{D}_k:=\{\dots<x^k_{-1}<x^k_0<x_1^k<\dots\}$ then $\mathcal{D}_{k+1}=
 \mathcal{D}_k\cup\{\displaystyle{x_{n}^{k+1}=\frac{x_{n+1}^k+x_n^k}{2}},
 n\in\mathbb{Z}\}$. 
 \end{itemize}
  Set $\mathcal{D}:=\displaystyle{\cup_k \mathcal{D}_k}$; $\mathcal{D}$ is the set of dyadic numbers in $\R$ and it is a dense set. 
 \begin{lemma}\label{2}
For all $x$ in $\mathcal{D}$, $f(x)=x$ and the horocycle $h_{x+i}(x)$ is fixed setwise by $f$.  
 \end{lemma}
 \begin{proof} First, $f(n)=n$ for all $n\in\mathcal{D}_0=\z$ and every horocycle $h_n^0$ centered at $n$ with diameter $1$ is fixed setwise.

Now, assume that $\mathcal{D}_k$ is fixed pointwise by $f$ and every horocycle $h_n^k$ of diameter ~$\frac{1}{2^k}$ and centered at $x_n^k\in \mathcal{D}_k$ is fixed setwise. For every $n\in\z$, $$h_n^k\cap h_{n+1}^k=\{z_n^k:=\frac{x_{n+1}^k+x_n^k}{2}+\frac{1}{2^{k+1}}i\}.$$ 
 Since $h_n^k$ is a fixed set of $f$ for all $n$, we have $f(z_n^k)=z_n^k$ for all $n$. Therefore, the horizontal $L_k$ line passing through all $z_n^k$ is a fixed set. Let $h_{z_n^k}(\frac{x_n^k+x_{n+1}^k}{2})$ be the horocycle passing through $z_n^k$ and centered at $\frac{x_n^k+x_{n+1}^k}{2}$ (see \Cref{seq}).  That horocycle is tangent to $L_k$ at $z_n^k$. Since $f(z_n^k)=z_n^k$ and $f(L_k)=L_k$, then  $h_{z_n^k}(\frac{x_n^k+x_{n+1}^k}{2})$ is fixed setwise and this implies that $f(\frac{x_n^k+x_{n+1}^k}{2})=\frac{x_n^k+x_{n+1}^k}{2}$. So, $f(x)=x$ for $x\in\mathcal{D}_{k+1}$ and the proof follows by induction.
 
  Let $x\in \mathcal{D}$ and $h_{x+i}(x)\in\Ho$. Then, $f(x)=x$ and $f(x+i)\in h_i(\infty)$. Since $h_i(\infty)$ is tangent to $h_{x+i}(x)$, $f(h_{i}(\infty))=h_{i}(\infty)$ is tangent to $f(h_{x+i}(x))$. It implies that $f(x+i)=x+i$ and $h_{x+i}(x)$ is a fixed set of $f$.   
 \end{proof}

 \begin{proposition}\label{ctr} Let $f:\h\rightarrow\h$ be a horocycle-preserving bijection. Then, there exists an isometry $j$ such that $(j\circ f)|_{\partial\h}\equiv \mathrm{Id}$, and $j\circ f$ fixes pointwise $h_i(\infty)$.
 \end{proposition}
 \begin{proof} \Cref{2} implies that there exist $j$ such that $j\circ f(x)=x$ for all $x\in\mathcal{D}$. Now, assume that $x\in\R-\mathcal{D}$. Set $A:=\mathcal{D}\cap((-\infty,x-1)\cup(x+1,\infty))$ and $\Ho_{A}$ denotes the set of all horocycles centered at a point in $A$ and tangent to $h_i(\infty)$.
 By \Cref{2}, for all $h\in\Ho_{A}$, $h$ is a fixed set of $j\circ f$. So, $j\circ f(h_{x+i}(x))=h_{x+i}(x)$ since $h_{x+i}(x)$ is pinched between elements of $\Ho_{A}$. Hence, $j\circ f(x)=x$. 
 
 On the other side, we already know that $L$ is fixed setwise by $j\circ f$. Let $z=x+i$ be a point in ~$L$. We have $j\circ f(x)=x$, $h_z(x)=f(h_z(x))=h_{f(z)}(x)$ and $f(z)\in f(L)=L$. It implies that $j\circ f(z)\in L\cap h_z(x)=\{z\}$ and this achieve the proof.
 \end{proof}
 Now, we give the proof of \Cref{main}.
 \begin{figure}[htbp]
\begin{center}
\includegraphics[scale=0.4]{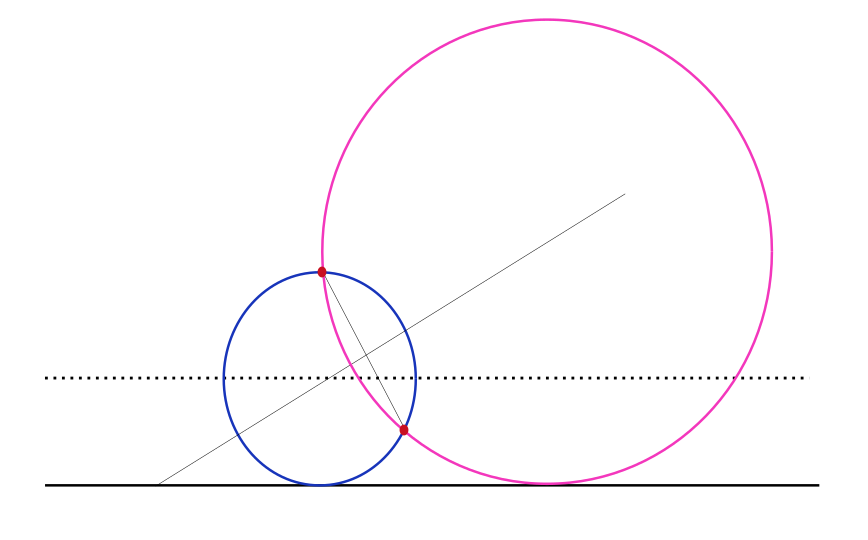}
\put(-102,20){\small{$z'$}}
\put(-118,55){\small{$z$}}
\put(-139,35){\small{$z_0$}}
\put(-88,35){\small{$z_1$}}
\caption{}
\label{default}
\end{center}
\end{figure}
\begin{proof}[\textbf{Proof of \Cref{main}}]
Let $f:\h\rightarrow\h$ be a bijection which sends horocycles to horocycles. By \Cref{projet} $f$ extends to $\partial\h$ and by \Cref{ctr} there exists $j\in\mathrm{Isom}(\h)$ such that $j\circ f$ fixes $\R$ and $L:=\{z\in\h, \mathrm{Im}(z)=~1\}$.

 Let $z=x+iy\in\h$ such that $y>1$ (the case $y<1$ follows the same idea). Then $h_z(x)\cap L:=\{z_0,z_1\}$. So $z_0$, $z_1$ and $x$ are fixed point of $j\circ f$. Hence, $h_z(x)$ is a fixed set since a horocycle is completely determined by three points. Let $z'\in h_z(x)$ such that $\mathrm{Im}(z')<1$ and $h_{z,z'}$ the horocylce passing through $z$ and $z'$: $j\circ f(h_{z,z'})=h_{z,z'}$. It follows $j\circ f(z)\in h_z(x)\cap h_{z,z'}=\{z,z'\}$. If $j\circ f(z)=z'$ then the order on the horocycles $h_z(\infty)$ and $h_i(\infty)$ is not preserved under $j\circ f$ which is absurd. So, $j\circ f(z)=z$. Therefore, $j\circ f=\mathrm{id}$ which achieve the proof. 
\end{proof}
\end{paragraph}

\section{Proof of \Cref{main2}}
Hypercylces intersect in four different ways and we first show that they are all preserved by $f$. This implies that hypercycles with same endpoints are mapped to hypercycles with the same property. Finally, we show that a bijection that sends hypercycles to hypercycles also sends geodesics to geodesics which achieves the proof by Jeffers theorem. 

A hypercycle  is an arc of circle (with endpoints on $ \partial\h $) not orthogonal to $\partial \mathbb{H}^2$ or a line in $\mathbb{H}^2$ that is neither vertical nor horizontal. Unlike geodesics, two hypercycles $h_1$ and $h_2$ may intersect in four different types:
\begin{itemize}
\item \textit{Type 1}: $h_1$ and $h_2$ intersect at one point and they are tangent;
\item \textit{Type 2:} $h_1$ and $h_2$ intersect at one point with one endpoint in common; 
\item \textit{Type 3:} $h_1$ and $h_2$ intersect at one point and have different endpoints;
\item \textit{Type 4:} $h_1$ and $h_2$ intersect at two points.  
\end{itemize} 
\begin{figure}[htbp]
\begin{center}
\includegraphics[scale=0.25]{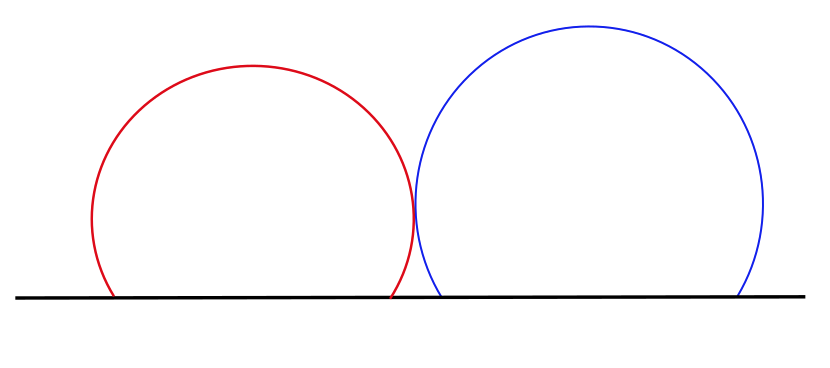}\put(-60,0){\tiny{\textbf{Type 1}}}\hspace{2cm}\includegraphics[scale=0.25]{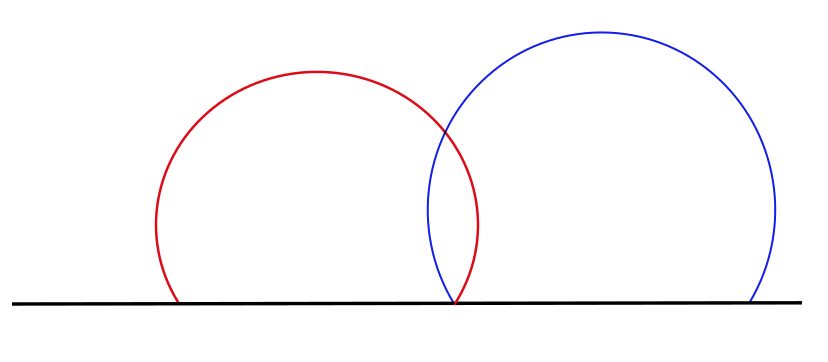}\put(-60,0){\tiny{\textbf{Type 2}}}\hspace{4cm}\\\includegraphics[scale=0.25]{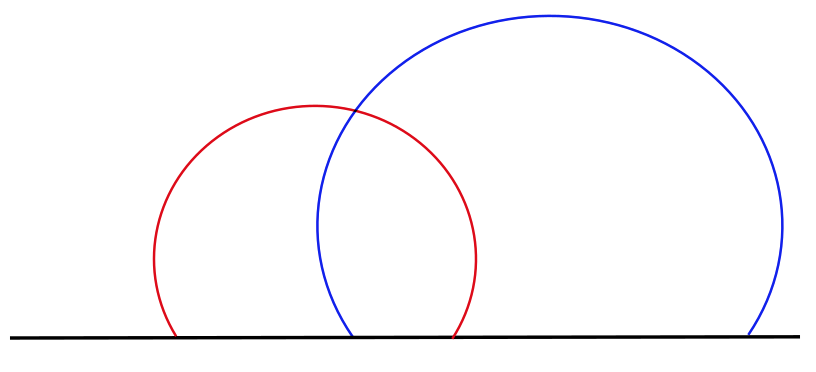}\put(-60,-3){\tiny{\textbf{Type 3}}}\hspace{2cm}\includegraphics[scale=0.25]{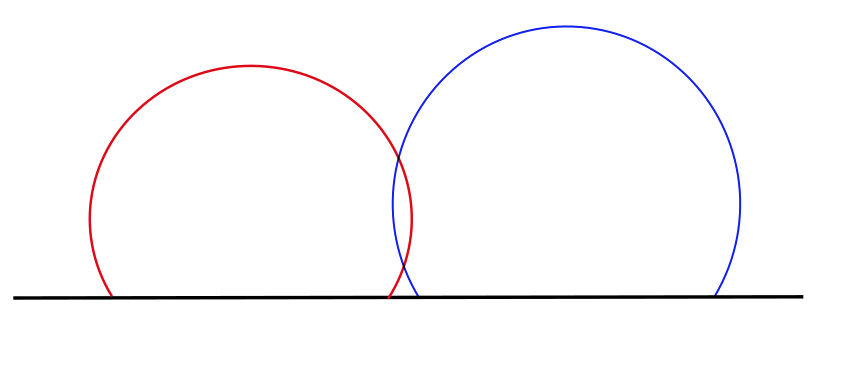}\put(-60,0){\tiny{\textbf{Type 4}}}
\caption{Different types of intersection between hypercycles}
\label{default}
\end{center}
\end{figure}
For the reminder of this section, we set $f$ to be a bijection that sends hypercycles to hypercylces. It is not hard to see that $f$ preserves intersections of type 4; that is if two hypercycles intersect twice so will be their images under $f$. We first start by showing that the other intersection types are also preserved. 

Let $h_1$ and $h_2$ be two hypercycles tangent at a point $p$. The point $p$ divided $h_1$ (respectively $h_2$) into two sub-arcs $I_1$ and $I_2$ (respectively $J_1$ and ~$J_2$). 
\begin{lemma}\label{hyp1} For all $x$ in $I_1$ and for all $y$ in $I_2$ there exists a hypercycle $h$ passing through $x$ and $y$ and disjoint to $h_2$. Moreover, this property holds only for pairs of hypercycles with intersection of Type 1.   
\end{lemma}
\begin{figure}[htbp]
\begin{center}
\includegraphics[scale=0.34]{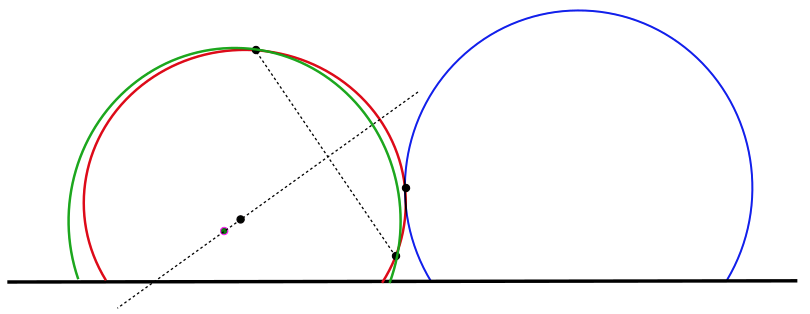}
\put(-67,23){\tiny{$p$}}
\put(-77,10){\tiny{$x$}}
\put(-98,48){\tiny{$y$}}
\put(-102,17){\tiny{$O$}}
\caption{}
\label{var1}
\end{center}
\end{figure}

\begin{proof}
Let $x$ and $y$ be two points in $I_1$ and $I_2$ respectively. Set $O$ to be the Euclidean center of the hypercycle $h_1$, it lies on the euclidean bisector of the segment $[x,y]$. Let $O'$ be a point in that bisector sufficiently close to $O$ such that its (Euclidean) distance to $[x ,y]$ is greater than the distance between $O$ and $[x, y]$. The circle centered at $O'$ passing through $x$ and $y$ defined a hypercycle passing through $x$ and $y$ and disjoint from $h_2$, since $p$ is a tangent point (see ~\Cref{var1}).   

In Type 2 and Type 3 configurations, $h_2$ divided $\mathbb{H}^2$ into two components. Moreover, $I_1$ and $I_2$ are in different components. So, every hypercycle $h$ from $x\in I_1$ to $y\in I_2$ crosses $h_2$.    
\end{proof}
The following is a direct consequence of \Cref{hyp1}:
\begin{corollary} The hypercycles $h_1$ and $h_2$ are tangent if and only if $f(h_1)$ and $f(h_{2})$ are. 
\end{corollary}
\begin{figure}[htbp]
\begin{center}
\includegraphics[scale=0.3]{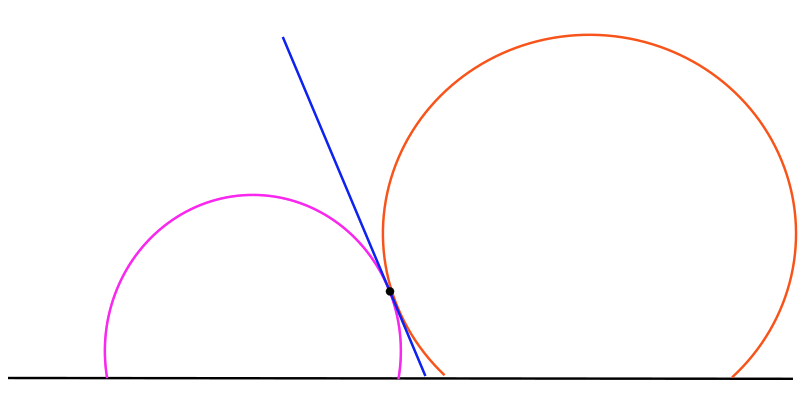}
\put(-100,32){\tiny{$h_1$}}
\put(-85,45){\tiny{$h_2$}}
\put(-30,58){\tiny{$h_3$}}
\caption{Three hypercycles $h_1$, $h_2$ and $h_3$ tangent at $p$ with $h_2$ between $h_1$ and $h_3$}
\label{default}
\end{center}
\end{figure}
We describe a special intersection pattern between three horocycles. 
\begin{definition}
Let $h_1$, $h_2$ and $h_3$ be three hypercycles tangent at a point $p$. We say that $h_2$ is \textit{between} of $h_1$ and $h_3$ if turning around $p$ on a small circle starting at $h_1$ gives $h_1-h_2-h_3-h_3-h_2-h_1$ as crossing pattern. Then, we say that $h_3$ is \textit{below} $h_2$ relatively to $h_1$. 
\end{definition}
\begin{lemma}
If $h_2$ is between of $h_1$ and $h_3$, then $f(h_2)$ is between $f(h_1)$ and $f(h_3)$. 
\end{lemma}
\begin{proof} Assume that $h_2$ is between $h_1$ and $h_3$. Then $h_2$ separates $h_1$ and $h_3$ and it follows that every hypercycle that intersect both $h_1$ and $h_3$ intersect $h_2$. Since $f$ preserves tangency,  $f(h_1)$, $f(h_2)$ and $f(h_3)$ are tangent at $f(p)$. If $f(h_2)$ is not between $f(h_1)$ and $f(h_3)$, then there exists a hypercycle $h$ that intersects $f(h_1)$ and $f(h_3)$ without intersecting $f(h_2)$. Therefore, $f^{-1}(h)$ is a hypercycle that intersects $h_1$ and $h_3$ and disjoint from $h_2$; which is a contradiction.  So,  $f(h_2)$ is between $f(h_1)$ and $f(h_3)$. 
\end{proof}

\begin{lemma} Let $h_1$ and $h_2$ be two hypercycles that intersect at a point of Type 2. Then, there exists a hypercycle $h$ tangent to $h_1$ at $p$ such that every hypercycle $h'$ below $h_1$ relatively to $h$ intersects $h_2$ twice. Moreover, this property does not occur for intersections of Type 3.   
\end{lemma}   
\begin{figure}[htbp]
\begin{center}
\includegraphics[scale=0.4]{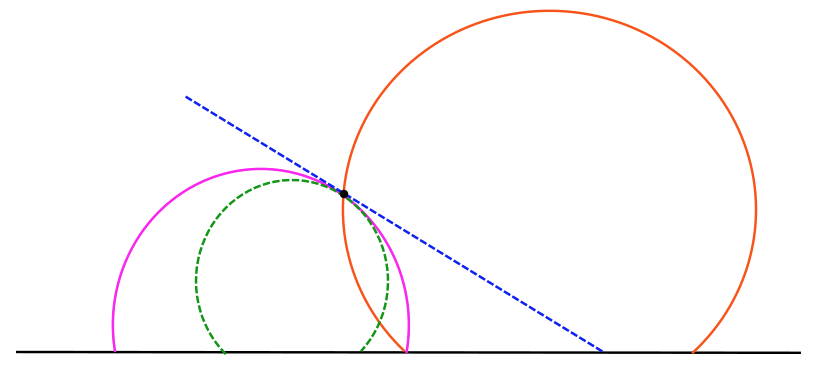}
\put(-143,0){\tiny{$x_1$}}
\put(-95,0){\tiny{$x_2=y_1$}}
\put(-27,0){\tiny{$y_2$}}
\put(-143,30){\tiny{$h_1$}}
\put(-10,30){\tiny{$h_2$}}
\put(-120,55){\tiny{$h$}}
\put(-120,25){\tiny{$h'$}}
\caption{}
\label{4tang}
\end{center}
\end{figure}
\begin{proof}
Let $h_1$ and $h_2$ be two hypercycles that intersect at point $p$ of Type ~2. Let $(x_1, x_2)$ and $(y_1, y_2)$ be the endpoints of $h_1$ and $h_2$ respectively; with $x_2=y_1$ (since $h_1$ and $h_2$ have one endpoint in common). Then, let $h$ be a hypercycle tangent to $h_1$ at $p$ and intersecting $h_2$ transversally at $p$ (see  \Cref{4tang}).

Then, if $h'$ is below $h_1$ relatively to $h$ its endpoints $(t_1, t_2)$ are in the interval $[x_1, x_2]$ that does not contained the endpoints of $h$. It follows that $(t_1, t_2)$ and $(y_1, y_2)$ are not intertwined and $h'$ intersects $h_2$. So, they intersect twice.

If $h_1$ and $h_2$ intersect once on Type 3 point, one can see that for every $h$ tangent to $h_1$, there exists $h'$ below $h_1$ relatively to $h$ that intersect $h_2$ at a transverse point; this come essentially from the fact that $h_1$ and $h_2$ have no common endpoints (see \Cref{5tang}).   
\begin{figure}[htbp]
\begin{center}
\includegraphics[scale=0.4]{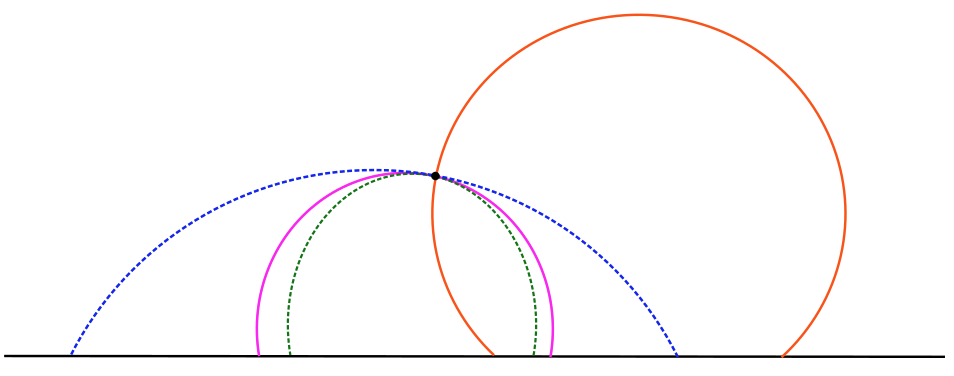}
\put(-20,30){\tiny{$h_2$}}
\put(-147,20){\tiny{$h_1$}}
\put(-131,13){\tiny{$h'$}}
\put(-172,20){\tiny{$h$}}
\caption{}
\label{5tang}
\end{center}
\end{figure}

\end{proof}
\begin{corollary} If $h_1$ and $h_2$ intersect at a point of Type 2 (respectively Type ~3), so are $f(h_1)$ and $f(h_2)$. 
\end{corollary}

Now that we know $f$ preserves the intersection types, we show the following:
\begin{proposition}\label{hyp2}
If $h_1$ and $h_2$ have the same endpoints, then $f(h_1)$ and $f(h_2)$  have also the same endpoints. 
\end{proposition}
\begin{proof} Assume that $h_1$ and $h_2$ have the same endpoints. So every hypercycles $h$ that intersect $h_1$ once at point of Type 3 also intersect $h_2$. 

If $f(h_1)$ and $f(h_2)$ have at most one endpoint in common, there exists $h$  that intersect $f(h_1)$ once at a point of Type 3 such that $h$ is disjoint from $f(h_2)$. By taking $f^{-1}(h)$ we obtain a hypercycle that intersect $h_1$ at a point of Type 3 and disjoint from $h_2$, which is impossible. So, $f(h_1)$ and $f(h_2)$ have the same endpoints.  
\end{proof}

Now, we are able to give the proof of \Cref{main2}. 
\begin{proof}[\textbf{Proof \Cref{main2}}]
Let us show that a hypercycle-preserving bijection $f$ maps geodesics to geodesics. Let $g$ be a geodesic in $\mathbb{H}^2$ with endpoints $(x_1, x_2)$ and $r\geq0$. The set of all points at distance $r$ from $g$ is a crescent defined by two hypercylces $h_1$ and $h_2$ with endpoints $(x_1, x_2)$. Then $f(h_1)$ and  $f(h_2)$ have the same endpoints by \Cref{hyp2}, and let $g'$ denotes the geodesic with same endpoints as $f(h_1)$ and $f(h_2)$. Assume that there is a point $x\in g$ such that $f(x)\notin g'$.  There exists a hypercycle $h'$ passing through $f(x)$ and with same endpoints as  $f(h_1)$ and $f(h_2)$. So, $f^{-1}(h')$ is a hypercycle with the same endpoints as $g$ and intersecting $g$ at $x$, and this is impossible. 

So, $f(g)=g'$ and it follows that $f$ sends geodesics to geodesics. Using Jeffers Theorem (\cite{Jef}-Theorem 2), we conclude that $f$ is an isometry.   
\end{proof}
\section{Proof of \Cref{ivan}}
We will prove \Cref{ivan} for the goedesic case and we will end up with the horocycle and hypercycle cases.

We recall that Thurston earthquake theorem states that every element of $\mathrm{Homeo}^+(\mathbb{S}^1)$ is the boundary map of an earthquake of $\h$. We will use it in our proof together with the followings:
\begin{lemma}\label{earth}
Let $f$ be an automorphism of $\mathcal{K}_{geod}$; and let $g_1$, $g_2$ be two disjoint geodesics with no common endpoints. Then $f(g_1)$ and $f(g_2)$ have no endpoints in common. Equivalently, if  $g_1$, $g_2$ are two disjoint geodesics with a common endpoint, so are $f(g_1)$ and $f(g_2)$.
\end{lemma}
\begin{proof}
Let $g_1$ and $g_2$ be two disjoint geodesics with no common endpoints --let's say $(x_1,x_2)$ and $(y_1,y_2)$ respectively-- with the following cyclic order $x_1<x_2<y_1<y_2$. Set $h_1$ and $h_2$ to be the geodesics defined by $(x_1,y_1)$ and $(x_2,y_2)$ respectively. It follows that the set $\{g_1,g_2,h_1,h_2\}$ has the following properties (see \Cref{exten}): 
\begin{itemize}
\item $g_1\cap g_2=\emptyset$, $h_1\cap h_2\neq\emptyset$ and $g_i\cap h_i=\emptyset$ for $i=1,2$;
\item for every geodesic $g$, $g\cap g_i\neq\emptyset\implies g\cap (h_1\cup h_2)\neq\emptyset$;
\item for every geodesic $g$, $|g\cap(g_1\cup g_2)|=2\implies |g\cap h_1|+ |g\cap h_2|=2$. 
\end{itemize}
\begin{figure}[htbp]
\begin{center}
\includegraphics[scale=0.5]{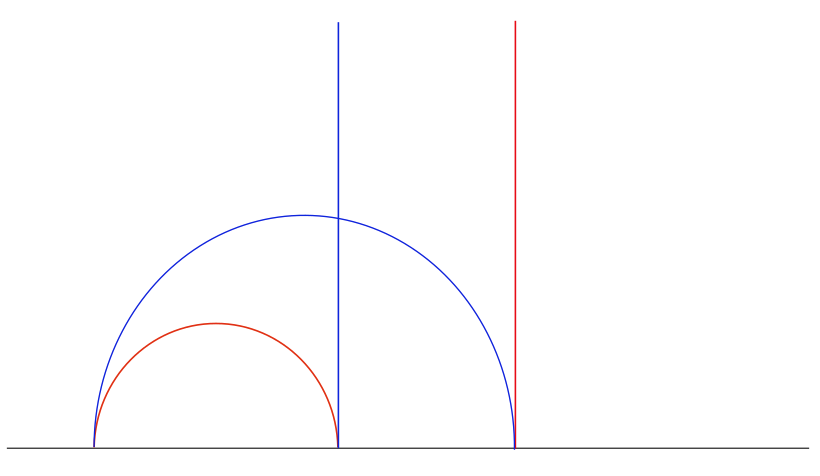}
\put(-150,40){\tiny{$g_1$}}
\put(-75,40){\tiny{$g_2$}}
\put(-120,40){\tiny{$h_2$}}
\put(-145,66){\tiny{$h_1$}}
\put(-185,02){\tiny{$x_1$}}
\put(-125,02){\tiny{$x_2$}}
\put(-79,02){\tiny{$y_1$}}
\put(-110,120){\tiny{$y_2=\infty$}}
\caption{{\color{red}}}
\label{exten}
\end{center}
\end{figure}

Therefore, the set $\{f(g_1),f(g_2),f(h_1),f(h_2)\}$ satisfy the same property. So, for $i=1,2$ if $f(g_i)$ has at most one endpoint in common with $f(h_1)\cup f(h_2)$, then there exists a geodesic $g$ intersecting $f(g_i)$ but disjoint from $f(h_1)\cup f(h_2)$; which is a contradiction. We deduce that for $i\in\{1,2\}$, $f(g_i)$ share one endpoint with $f(h_1)$ and the other one with $f(h_2)$. 

Now, if $f(g_1)$ and $f(g_2)$ has one endpoint in common, there exists a geodesic $g$ such that $|g\cap(f(g_1)\cup f(g_2))|=2$ but $|g\cap f(h_1)|+|g\cap f(h_2)|=1$ which also leads to a contradiction. In conclusion, we proved that $f(g_1)$ and $f(g_2)$ have different endpoints. 
\end{proof}

\begin{proof}[Proof of \Cref{ivan}: Geodesic case] We will show that an automorphism of the geodesic graph  is induced by an earthquake maps after postcomposing by an isometry. 

 Let's show that $\hat{f}\in\mathrm{Aut}(\mathcal{K}_{geod})$ defined a map on the boundary. Consider $x\in\mathbb{S}^1$; $g_1$ and $g_2$ be two disjoint geodesics with $x$ as a common endpoint. So, $\hat{f}(g_1)$ and $\hat{f}(g_2)$ are disjoint. Moreover $\hat{f}(g_1)$ and $\hat{f}(g_2)$  have a common endpoint $y$ by \Cref{earth}; set $\partial \hat{f}(x):=y$. The map $\partial\hat{f}$ is well-defined and is a bijection. Let $(x_n)$ be an increasing sequence with $x$ as limit. The sequence $g_n:=(x,x_n)$ of geodesics with endpoints $x$ and $x_n$ is such that $g_{n}$ is between $g_{n-1}$ and $g_{n+1}$, which mean every geodesic $g$ intersecting $g_{n-1}$ and $g_{n+1}$ also intersect $g_n$. Moreover, there is no geodesic $g$ with one endpoint equal to $x$ such that every $g_n$ is between $g$ and $g_0$. It follows that $h_n:=\hat{f}(g_n)=(\partial\hat{f}(x),\partial\hat{f}(x_n))$ is a sequence of geodesics with a common endpoint such that $h_{n}$ is between $h_{n-1}$ and $h_{n+1}$. This implies that the sequence $(\partial\hat{f}(x_n))$ is either increasing or decreasing; and let $l:=\lim \partial\hat{f}(x_n)$. If $l\neq \partial\hat{f}(x)$, then the geodesic $h$ with endpoint $(\partial\hat{f}(x),l)$ is such that $h_n$ is between $h_0$ and $h$ for all $n$, which leads to a contradiction. So, $l=\partial\hat{f}(x)$ and we have continuity for $\partial\hat{f}$. Hence, $\partial\hat{f}$ is a homeomorphism. 
Now let's show that after postcomposing with an isometry, $\partial\hat{f}$ is orientation-preserving. Since the action of isometries on $\mathbb{S}^1$ is 3-transitive, after postcomposing with an isometry we can assume that $\partial\hat{f}$ fixes $-1$, $1$ and $\infty$. Let $x_1$, $x_2$  be two points on $\mathbb{S}^1-[-1,1]$ such $x_1<x_2$ (where "$<$" stands for the counterclockwise order on $\mathbb{S}^1$) and assume that $\partial\hat{f}(x_2)<\partial\hat{f}(x_1)$. Then $g_1:=(1,x_1)$ and $g_2:=(-1,x_2)$ are disjoint but not $\hat{f}(g_1)=(1,\partial\hat{f}(x_1))$ and $\hat{f}(g_2)=(-1,\partial\hat{f}(x_2))$, which is absurd since $\hat{f}$ is an automorphism. The other cases follow the same idea since there exists a segment $I\in\{[-1,1], [1,\infty], [\infty, -1]\}$ which does not contained $x_1$ and $x_2$. Thus $\partial\hat{f}\in\mathrm{Homeo}^+(\mathbb{S}^1)$ and by Thurston earthquake theorem, $\partial\hat{f}\in\mathrm{Homeo}^+(\mathbb{S}^1)$ is induced by an earthquake map.   
\end{proof}

Let's introduce some terminology that are going to be useful for the proof in horocycle and hypercycle cases. 
\begin{figure}[htbp]
\begin{center}
\includegraphics[scale=0.5]{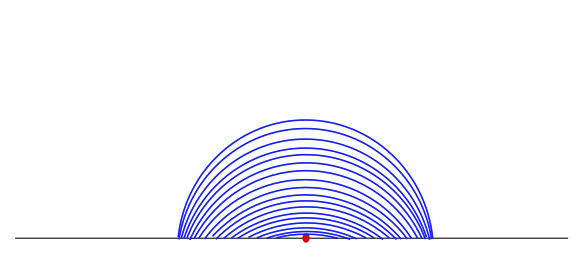}\hspace{0.6cm}\includegraphics[scale=0.4]{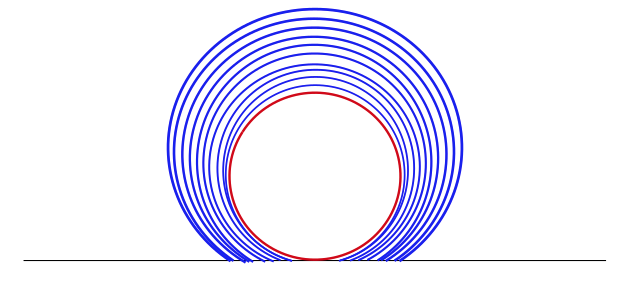}\hspace{0.6cm}\includegraphics[scale=0.5]{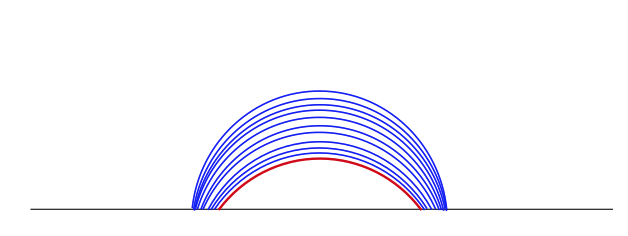}
\caption{Three different types of continuous family of hypercycles converging to a point, a horocycle and a hypercycle.}
\label{lime}
\end{center}
\end{figure}
\begin{definition}
A family $\{h_t\}_{t\in[0,1]}$ of hypercycles (respectively horocycles) is \textit{continuous} if:
\begin{itemize}
\item $h_l$ is between $h_t$ and $h_s$, i.e, every hypercycles (respectively horocycles) intersecting $h_s$ and $h_t$ intersects $h_l$ whenever $t<l<s$;
\item if $h$ is a hypercycle (respectively horocycle) between $h_s$ and $h_l$, then either $h$ intersects uncountably many $h_t$
's or there exists $t\in[0,1]$ such that $h=h_t$ (respectively there exists $t\in[0,1]$ such that $h=h_t$). This means that there is no gap in the family $\{h_t\}$.  
\end{itemize}
A continuous family $\{h_t\}_{t\in[0,1]}$ \textit{exhausts} a continuous family $\{h'_t\}_{t\in[0,1]}$ if for every $t\in[0,1]$, there exists $t_0\in[0,1]$ such that  $h_s$ is between $h'_t$ and $h_l$ for all $l>s>t_0$. 
\end{definition}

By definition, elements in a continuous family of horocycles must have the same center. A horocycle $h$ can be seen as the limit set of a continuous family $\{h_t\}$ of hypercycles. Even more, we have:
\begin{lemma}\label{disj}
 A horocyle $h$ is disjoint from a hypercycle $h'$ if and only if there exists a continuous family $\{h'_t\}$ of hypercycles converging to $h$ and such that $h'_0=h'$.
 \end{lemma}
\begin{proof} Since isometries acts transitively on horocylces, we can assume that $h$ is centered at $\infty$ and passes through $i$. The hypercyles $h'$ is below $h$ and up to a translation, we can assume that $h'$ intersects the $y$-axis orthogonally at $ie^{-a}$ with $a>0$. Let $b$ and $-b$ be the endpoints of $h'$. We defined $h_t$ to be the hypercycle with endpoints $\{-b^{\frac{1}{1-t}},b^{\frac{1}{1-t}}\}$ and passing through $z_t:=ie^{a(t-1)}$ for $t\in[0,1)$. The set $\{h_t\}_{t\in[0,1]}$ is a continuous family of hypercycles that accumulate to the horizontal horocycle $h$ passing through $i$.   
\end{proof}

\begin{figure}[htbp]
\begin{center}
\includegraphics[scale=0.3]{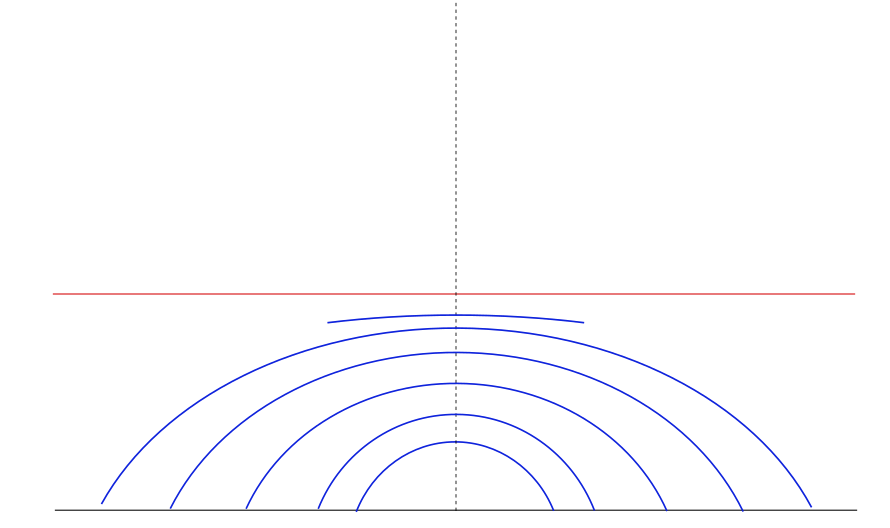}
\caption{}
\label{}
\end{center}
\end{figure}

Since continuous families of horocycles/hypercycles are defined using intersection properties, it follows that they are preserved under automorphisms, i.e, an automorphism sends a continuous family of horocycles/hypercycles to a continuous family.  As consequence, we have:

\begin{lemma}\label{gent} Let $\hat{f}$ be an automorphism of $\mathcal{K}_{horo}$ (respectively $\mathcal{K}_{hyper}$): 
\begin{enumerate}
\item If $h_1$ and $h_2$ are two tangent horocycles, then $\hat{f}(h_1)$ is tangent to $\hat{f}(h_2)$. 
\item If $h_1$ and $h_2$ are two hypercycles with the same endpoints, so are $\hat{f}(h_1)$ and $\hat{f}(h_2)$.   
\end{enumerate}
\end{lemma}
\begin{proof}
Let $\hat{f}$ be an automorphim and assume that $h_1$ and $h_2$ are two tangent horocycles. So, there exists a continuous family $\{h^1_t\}$ such that $h_0^1=h_1$ and $h^1_t\cap h_2=\emptyset$ for all $t>0$. The set $\{\hat{f}(h^1_t)\}$ is a continuous family where $\hat{f}(h^1_0)=\hat{f}(h_0)$ intersects $\hat{f}(h_2)$ and $\hat{f}(h^1_t)$ is disjoint from $h_2$ for all $t>0$. Hence, $\hat{f}(h_1)$ is tangent to $\hat{f}(h_2)$. 

Two hypercycles $h_1$ and $h_2$ have the same endpoints if and only if there exists a hypercycle $h_3$ between $h_1$ and $h_2$ which is characterized by an intersection property. Since the property of having the same endpoints is characterized by intersection, it is preserved by automorphisms.  
\end{proof}
\begin{lemma}\label{lim}
Let $\{h_t\}$ be a continuous family of hypercycles converging to a horocycle. Then, we have one of the following situations:
\begin{itemize}
\item the union of all $h_t$ is equal to one of the components of $\h-h_0$;
\item $\{h_t\}$ converges to a horocycle or a hypercycle/geodesic. 
\end{itemize}
Moreover, an automorphism of $\mathcal{K}_{hyper}$ preserves the convergence type of a family $\{\hat{f}(h_t)\}$. 
\end{lemma}
\begin{proof}
Assume that $\{h_t\}$ does not foliate any component of $\h-h_0$. Let $\{x_t^-\}$ and $\{x_t^+\}$ be the set of endpoints of the family $\{h_t\}$, i.e, $\{x_t^-,x_t^+\}$ are the endpoints of $h_t$ where $x_t^-<x_t^+$. By defintion of a continuous family, the set $\{x_t^-\}$ is bounded from above by any element $x_t^+$ and $\{x_t^+\}$ is bounded from below by any $x_t^-$; set $\sup\{x_t^-\}=m$ and $\inf\{x_t^+\}:=M$. There are two cases depending on whether $m=M$ or not. If $m=M$, there exists a horocycle $h$ centered at $m$ and disjoint from $h_t$ for all $t$, since $\{h_t\}$ foliates any component of $\h-h_0$. We denote by $\mathcal{H}_m$ the set of all horocycles centered at $m$ and disjoint from $h_t$ for all $t$ and let $h'$ be the one with maximal radius. Since $h'$ has maximal radius, it follows that $\{h_t\}$ converges to $h'$.  The case $m\neq n$ follows the same idea by taking the set $\mathcal{H}_{hyper}$ of all hypercycles disjoint from the $h_t$'s with endpoints $m$ and $M$. 

Now, assume that $\{h_t\}$ is a continuous family of hypercycles and that $\{h_t\}$ foliates one of the components of $\h-h_0$. Then, for every hypercycle in the foliated component with the same endpoints as $h_0$ there exist $h_t$ such that $h$ is between $h_0$ and $h_t$. On the other hand, if $\{h_t\}$ is a continuous family admitting a hypercycle $h'$ as a limit, $h'$ is a hypercycle in the component $\h-h_0$ containing the $h_t$'s and $h'$ is disjoint from $h_t$ for all $t$. Finally, if $\{h_t\}$ converges to a horocycle $h'$ there exists a hypercycle with the same endpoints as $h_0$ and intersecting all $h_t$ ( take among the ones that intersects the horocycle $h'$) and there is no hypercycle disjoint from all $h_t$ in the connected component of $\h-h_0$ containing all $h_t$.\\
Therefore, the three types of convergence are characterized by different types of intersection patterns, so they are preserved by automorphisms.      
\end{proof}

\begin{proof}[\textbf{Proof of \Cref{ivan}:} Horocycles/Hypercycles cases]   
Let $\hat{f}$ be in $\mathrm{Aut}(\mathcal{K}_{horo})$; $h_0$ the horocycle centered at $0$ and passing through $i$ and $h_1$ the horizontal horocycle passing through $i$. Therefore, $\hat{f}(h_0)$ and $\hat{f}(h_1)$ are tangent, thanks to \Cref{gent}.  Since $\mathrm{Isom}(\h)$ is 3-transitive on $\mathbb{S}^1$, we can assume that $\hat{f}(h_0)$ and $\hat{f}(h_1)$ have the same center like $h_0$ and $h_1$, after postcomposing with an isometry. So, $\hat{f}(h_0)$ and $\hat{f}(h_1)$ are tangent at $ia$ and again by postcomposing with $z\mapsto \frac{z}{a}$ we have $\hat{f}(h_0)=h_0$ and $\hat{f}(h_1)=h_1$. From here, we can use the same argument like in the proof of \Cref{2} to show that every horocycle centered at $x$ and passing through $x+i$ is fixed by $\hat{f}$.\\
Now if $h$ is a given horocycle, there exists two horocycles $h'$ and $h''$ tangent to $h_0$ and such that $h'$ and $h''$ are tangent to $h$ (see \Cref{pince}). 
\begin{figure}[htbp]
\begin{center}
\includegraphics[scale=0.7]{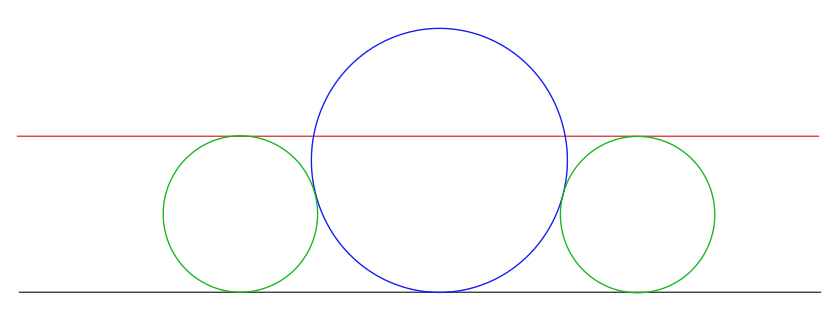}
\put(-140,105){\tiny{$h$}}
\put(-245,30){\tiny{$h'$}}
\put(-42,30){\tiny{$h''$}}
\put(-280,68){\tiny{$h_0$}}
\caption{}
\label{pince}
\end{center}
\end{figure}
Since $\hat{f}$ fixes $h'$ and $h''$, it follows that $\hat{f}$ fixes $h$. Thus, after postcomposing with an isometry $\hat{f}$ acts trivially on $\mathcal{K}_{horo}$ and this implies that $\hat{f}$ is induced by an isometry.   

Now, let $\hat{f}$ be an automorphism of $\mathcal{K}_{hyper}$. Our goal is to show that $\hat{f}$ induced an automorphism $\hat{f}^*$ of $\mathcal{K}_{horo}$. Let $h$ be an horocycle and $\{h_t\}$ be a continuous  family of hypercycles with limit $h$. By \Cref{lim} $\{\hat{f}(h_t)\}$ is a continuous family which converges to a horocycle $h'$; we defined $\hat{f}^*(h):=h'$. We claim that $\hat{f}^*$ is an automorphism. Consider $h_1$ and $h_2$ two disjoint horocycles with $\{h^1_t\}$ and  $\{h^2_t\}$ two continuous family of hypercycles converging to $h_1$ and $h_2$ respectively. Then, there exists $t_0$ such for $s,l\geq t_0$, $h^1_s\cap h^2_l=\emptyset$. And this implies that $\hat{f}^*(h_1)$ is disjoint from $\hat{f}^*(h_2)$, and thus $\hat{f}^*$ is an automorphism of $\mathcal{K}_{horo}$. Let $\phi:\mathrm{Aut}(\mathcal{K}_{hyper})\rightarrow \mathrm{Aut}(\mathcal{K}_{horo})$ defined by $\phi(\hat{f}):=\hat{f}^*$; $\phi$ is injective. Assume that $\hat{f}\neq \mathrm{Id}$ while $\hat{f}^*=\mathrm{Id}$. Then, there exists a hypercycle $h$ such that $\hat{f}(h)\neq h$; and a horocycle $h_0$ disjoint form $h$ such that $\hat{f}(h)\cap h_0\neq \emptyset$. Since $h$ is disjoint from $h_0$, by \Cref{disj} there exists a continuous family $\{h_t\}$ containing $h$ such that $h_0$ is its limit set. It follows that $\{\hat{f}(h_t)\}$ is a continuous family converging to $\hat{f}^*(h_0)=h_0$; which is absurd since $\{\hat{f}(h_t)\}$ contains $\hat{f}(h)$ which is not disjoint form $h_0$. Hence, $\phi$ is injective and this implies that every element of $\mathrm{Aut}(\mathcal{K}_{hyper})$ is induced by an isometry, as $\mathrm{Aut}(\mathcal{K}_{hyper})$ can be seen as a subgroup of $\mathrm{Aut}(\mathcal{K}_{horo})$. 
\end{proof}

\bibliography{references}{}
\bibliographystyle{alpha}\vspace{1cm}

\noindent\textit{\textbf{cheikh21.lo@ucad.edu.sn}}, University Cheikh Anta Diop, Dakar, Senegal\\
\textbf{\textit{abdoulkarim3.sane@ucad.edu.sn}}, University Cheikh Anta Diop, Dakar, Senegal 
\end{document}